\documentclass[a4paper,11pt]{article}

\usepackage{amsmath,amssymb,bm,amsthm}

\usepackage{xcolor}
\definecolor{mediumgreen}{RGB}{0,128,0}

\usepackage[dvipdfmx]{graphicx}
\usepackage{here}

\usepackage{url}
\usepackage{comment}
\newtheorem{thm}{Theorem}[section]
\theoremstyle{plain}
\newtheorem{lemma}{Lemma}
\newtheorem{prop}[lemma]{Proposition}

\title{On the Structure of Multilinear Invariants of a Finite Unitary Reflection Group }

\author{A.K.M. Selim Reza\footnote{Corresponding author},  Manabu Oura, and Masashi Kosuda}

\begin{document}

\date{\empty} 
\maketitle

\textbf{Abstract}:  
We study the space of multilinear invariants \( V_f \) of degree \( f \) for a specified finite unitary reflection group. A subspace \( W_f \) of typical invariants is also introduced. We note that the dimension of \( W_f \) is given by Catalan number. We explore both spaces for \( f \leq 5 \), noting that their dimensions differ based on the value of \( f \). 
We explicitly determine the bases for both spaces, and then we establish the relationship between the vectors of the two bases.



\section{Introduction}

In this study, we focus on the action of a finite subgroup of the general linear group \( GL(2, \mathbb{C}) \), denoted by \( \mathcal{G} \), which is generated by the matrices
\[
 T = \frac{1+i}{2} \begin{pmatrix} 1 & 1 \\ 1 & -1 \end{pmatrix}, \quad D = \begin{pmatrix} 1 & 0 \\ 0 & i \end{pmatrix}.
\]
The group \( \mathcal{G} \) has order 96 and is one of the 37 classes of complex reflection groups. Specifically, it corresponds to the 8th group in the Shephard-Todd classification \cite{shephard-todd}. In the earlier work \cite{kosuda-oura},  the second and third named authors investigated the centralizer ring of tensor products of \( \mathcal{G} \). Building on that foundation, the present paper focuses on the detailed computation of multilinear invariants under the action of the group \( \mathcal{G} \). 

We define two distinct spaces of invariants: \( V_f \), the space of multilinear invariants of the group \( \mathcal{G} \) of degree \( f \), and \( W_f \), a subspace spanned by typical invariants, which are a particular type of multilinear forms that will be introduced in the next section. All multilinear invariants are constructed as linear combinations of column and row vectors, transforming according to specific rules under the action of the group \( \mathcal{G} \). This leads to the natural question: is \( W_f \) equal to \( V_f \)? We limit our investigation to \( f \leq 5 \), as computations for larger \( f \) become too large and time-consuming. For \( f = 1, 2, 3 \), we have that \( W_f \) and \( V_f \) coincide, while for \( f = 4 \) and \( f = 5 \), \( W_f \) has a smaller dimension than \( V_f \). In these cases, we identify the additional invariants needed to complete the basis of \( V_f \). We also establish the relationship between the vectors of the two bases of \( W_f \) and \( V_f \).

\medskip

In the next section, we provide the basic preliminaries to understand the transformation behavior of column and row vectors under the action of \( \mathcal{G} \). All computations in this work are performed using SageMath \cite{sagemath} and Maple.

\section{Preliminaries}
There is a renowned book by H. Weyl \cite{weyl} that extensively explores both invariant theory and representation theory, laying the groundwork for much of the understanding of symmetries in mathematics. In defining multilinear invariants, we follow his approach, which provides a systematic method for constructing invariants under group actions.

\noindent Let \( y^{(1)}, \dots, y^{(f)} \) denote column vectors, and \( \xi^{(1)}, \dots, \xi^{(f)} \) denote row vectors.
For all \( A \in \mathcal{G} \) we consider the following transformations :
\[
y^{(i)} \mapsto  A \, y^{(i)}, \quad \xi^{(i)} \mapsto  \xi^{(i)} A^{-1} \quad \text{for } i = 1, \dots, f.
\]

This dual transformation is fundamental in the construction of group-invariant tensors, as it ensures the invariance of certain forms under transformations. In this framework, one naturally encounters arrays of components
\[
\| b(i_1, \dots, i_f; k_1, \dots, k_f) \|,
\]
which represent the coefficients of the invariant forms. These coefficients determine how the column and row vectors combine to form invariants under the group action. Specifically, these invariants take the form
\begin{multline}
L(\xi^{(1)}, \xi^{(2)}, \dots, \xi^{(f)}; y^{(1)}, y^{(2)}, \dots, y^{(f)}) \\
= \sum b(i_1, i_2, \dots, i_f; k_1, k_2, \dots, k_f) \, \xi^{(1)}_{i_1} \cdots \xi^{(f)}_{i_f} \, y^{(1)}_{k_1} \cdots y^{(f)}_{k_f}.
\end{multline}

We say that \( L \) is a multilinear invariant of the group \( \mathcal{G} \subset GL(2, \mathbb{C}) \) if, for all \( A \in \mathcal{G} \), the following condition holds:
\begin{multline}
L(\xi^{(1)} A^{-1}, \dots, \xi^{(f)} A^{-1}; A y^{(1)}, \dots, A y^{(f)}) \\
= L(\xi^{(1)}, \dots, \xi^{(f)}; y^{(1)}, \dots, y^{(f)}).
\end{multline}

In our computations, we adopt the average form outlined in the following proposition, which simplifies the calculation of invariants while ensuring their consistency under the group action of \( \mathcal{G} \).

\begin{prop}
\label{multi}
Let 
\[
\begin{aligned}
L_{(i_1, i_2, \ldots, i_f; \, k_1, k_2, \ldots, k_f)} 
&\big( \xi^{(1)}, \xi^{(2)}, \ldots, \xi^{(f)}; \, y^{(1)}, y^{(2)}, \ldots, y^{(f)} \big) \\
&:= \prod_{j=1}^f \xi^{(j)}_{i_j} \, y^{(j)}_{k_j}, \quad
i_j, k_j \in \{1, 2\} \text{ for } j=1,\ldots,f.
\end{aligned}
\]

be a multilinear form on vectors \(\xi^{(j)}, y^{(j)} \in \mathbb{C}^2\). Define the averaged form

\begin{equation}
\begin{split}
\tilde{L}(\xi^{(1)}, \dots, \xi^{(f)}; y^{(1)}, \dots, y^{(f)}) := & \frac{1}{|G|} \sum_{A \in G} L\big(\xi^{(1)} A^{-1}, \dots, \xi^{(f)} A^{-1}; \\
& A y^{(1)}, \dots, A y^{(f)}\big).
\end{split}
\label{eqn:Vf}
\end{equation}

Then \(\tilde{L}\) is invariant under the action of the group \(G \subset GL(2,\mathbb{C})\)
\end{prop}

\begin{proof}

We want to show that \(\tilde{L}\) is invariant under the action of any \(B \in G\), i.e.
\[
\tilde{L}(\xi^{(1)} B^{-1}, \dots, \xi^{(f)} B^{-1}; B y^{(1)}, \dots, B y^{(f)}) = \tilde{L}(\xi^{(1)}, \dots, \xi^{(f)}; y^{(1)}, \dots, y^{(f)}).
\]

By definition,
\begin{multline*}
\tilde{L}(\xi^{(1)} B^{-1}, \dots, \xi^{(f)} B^{-1}; B y^{(1)}, \dots, B y^{(f)}) \\
= \frac{1}{|G|} \sum_{A \in G} L\big(\xi^{(1)} B^{-1} A^{-1}, \dots, \xi^{(f)} B^{-1} A^{-1}; A B y^{(1)}, \dots, A B y^{(f)}\big).
\end{multline*}

Since matrix multiplication is associative, we rewrite the arguments inside \(L\) as
\[
\xi^{(j)} B^{-1} A^{-1} = \xi^{(j)} (A B)^{-1} \quad \text{and} \quad A B y^{(j)} = (A B) y^{(j)} \quad \text{for each } j=1,\dots,f.
\]

Now perform the change of variables \(C = A B\). Because \(G\) is a group, the map \(A \mapsto C = A B\) is a bijection from \(G\) to \(G\). Thus, the sum over \(A \in G\) is equal to the sum over \(C \in G\):
\begin{multline*}
\frac{1}{|G|} \sum_{A \in G} L\big(\xi^{(1)} (A B)^{-1}, \dots; (A B) y^{(1)}, \dots \big) \\
= \frac{1}{|G|} \sum_{C \in G} L\big(\xi^{(1)} C^{-1}, \dots, \xi^{(f)} C^{-1}; C y^{(1)}, \dots, C y^{(f)} \big).
\end{multline*}

But this is precisely the definition of \(\tilde{L}(\xi^{(1)}, \dots, \xi^{(f)}; y^{(1)}, \dots, y^{(f)})\). Therefore,
\[
\tilde{L}(\xi^{(1)} B^{-1}, \dots, \xi^{(f)} B^{-1}; B y^{(1)}, \dots, B y^{(f)}) = \tilde{L}(\xi^{(1)}, \dots, \xi^{(f)}; y^{(1)}, \dots, y^{(f)}),
\]
which proves that \(\tilde{L}\) is an invariant of the group \(G\).
\end{proof}

In this context , we consider \( V_f \) as the complex vector space generated by multilinear invariants of degree \( f \) under the action of the group \( \mathcal{G} \). This space \( V_f \) coincides with the centralizer algebra \( A_f = \operatorname{End}_{\mathcal{G}}(V^{\otimes f}) \), where \( V \) is the natural representation of \( \mathcal{G} \). A key result by Kosuda and Oura~\cite{kosuda-oura} provides a closed-form expression for this dimension.

\begin{thm}[\cite{kosuda-oura}]

$\dim V_f=2^{f-2}+\dfrac{2^{2f-3}}{3}+\dfrac{1}{3}$ for each \( f \geq 1 \).
\label{thm:oura}
\end{thm}

\noindent As discussed in Dieudonné and Carrell \cite{dieudonne}, classical invariant theory provides a fundamental framework for understanding the transformation properties of vectors under the action of general linear groups.  
We follow them to construct another form of multilinear invariants. Since $\mathcal{G}$ is size of $2 \times2$ matrix, therefore we  construct those invariants under the action of \(\mathrm{GL}(2, \mathbb{C})\).

We consider a row vector
\[
\xi = \begin{pmatrix} \xi_1 & \xi_2 \end{pmatrix}
\]
and a column vector
\[
y = \begin{pmatrix} y_1 \\ y_2 \end{pmatrix},
\]
both elements of the two-dimensional complex vector space \( \mathbb{C}^2 \) and its dual, respectively. The general linear group \( \mathrm{GL}(2, \mathbb{C}) \) acts naturally on these vectors, where the row vector \( \xi \) transforms via the inverse of the group element, while the column vector \( y \) transforms directly. More explicitly, for any \( A \in \mathrm{GL}(2, \mathbb{C}) \), the transformation laws are given by
\[
\xi \mapsto \xi A^{-1}, \quad y \mapsto A y.
\]
These transformation rules reflect the classical distinction between row and column components under a change of basis. A fundamental consequence of these transformation properties is the invariance of the scalar product
\[
\langle \xi, y \rangle := \xi y,
\]
which remains unchanged under the simultaneous action of \( A \):
\[
(\xi A^{-1})(A y) = \xi y.
\]

This invariance exemplifies a core principle in classical invariant theory. Extending this concept, we define multilinear typical invariants by considering products of such scalar invariants across multiple vector pairs. Specifically, for an integer \( f \), let \( \xi^{(1)}, \dots, \xi^{(f)}\) be row vectors, and \( y^{(1)}, \dots, y^{(f)}\) be column vectors. For any permutations \( \alpha = (\alpha_1, \ldots, \alpha_f) \) and \( \beta = (\beta_1, \ldots, \beta_f) \) of \( \{1, 2, \ldots, f\} \), we define the invariant
\begin{equation}
\mathcal{W}(\alpha_1, \dots, \alpha_f;\, \beta_1, \dots, \beta_f)
= \prod_{k=1}^f \langle \xi^{(\alpha_k)},\, y^{(\beta_k)} \rangle.
\label{eqn:G}
\end{equation}

Each factor \( \langle \xi^{(\alpha_k)}, y^{(\beta_k)} \rangle \) is individually invariant under the group action, and therefore their product \( \mathcal{W} \) defines a multilinear invariant. Explicitly, if we write \( \xi^{(j)} = (\xi^j_1, \xi^j_2) \) and \( y^{(k)} = \begin{pmatrix} y^k_1 \\ y^k_2 \end{pmatrix} \), then the scalar product becomes
\[
\langle \xi^{(j)}, y^{(k)} \rangle = \sum_{i=1}^2 \xi^j_{i} y^k_{i},
\]
and equation~\eqref{eqn:G} takes the coordinate form
\begin{equation}
\mathcal{W}(\alpha; \beta)
= \prod_{k=1}^f \left( \sum_{i=1}^2 \xi^{\alpha_k}_ i y^{\beta_k}_ i \right).
\label{eqn:Wf}
\end{equation}

This construction produces a family of \( f!^2 \) bilinear invariants, one for each pair of permutations \( (\alpha, \beta) \), and forms the foundation for more general constructions in polynomial invariant theory under group actions.

We consider \( W_f \) as the complex vector space generated by all such invariants \( \mathcal{W}(\alpha_1, \dots, \alpha_f; \beta_1, \dots, \beta_f) \). Each of these invariants  satisfies the invariance condition under the action of \( \mathcal{G} \). Consequently, \( W_f \subset V_f \), where \( V_f \) is the space of all multilinear invariants of degree \( f \) under \( \mathcal{G} \).

\begin{thm}[\cite{Diehl}]
The dimension of $W_f$ coincides with the Catalan number, which is defined as:
\[
W_f = \frac{1}{f+1} \binom{2f}{f}
\]    
\end{thm}

\section{Results}

\begin{prop}
    The spaces \( W_f \) and \( V_f \) coincide for \( f = 1, 2, 3 \), meaning that the typical invariants \( W_f \) form a basis for the full invariant space \( V_f \).
\end{prop}

\noindent To determine the basis and dimension of \(W_f\) defined in equation (\ref{eqn:Wf}), we consider the set of all invariants and examine their linear independence. We then construct a matrix whose rows correspond to the invariants of \(\mathcal{W}\) and whose columns correspond to the distinct monomials appearing in their expansions. Applying Gaussian elimination to this matrix yields its rank, which equals the dimension of the space \(W_f\) and to get the basis we select those invariants that contribute to the rank of the matrix. Similarly, to determine the basis of \(V_f\) using the defined equation (\ref{eqn:Vf}) we generates all possible monomials, averages them over the group, and identifies the independent invariants by analyzing the rank of a coefficient matrix. To establish the relation between the basis vectors of \( W_f \) and \( V_f \), we explicitly address each case individually in the following :\\

\noindent When \( f = 1 \): 
    \[
    \dim V_1 = \dim W_1 = 1,
    \]    
and the basis vector \( N_1=\mathcal{W}(1;1) \) of \( W_1 \) has the expression:
    \[
    N_1 = 2  \tilde{L}_{1,1}
    \]
When \( f = 2 \): \[
    \dim V_2 = \dim W_2 = 2,
    \] 
The space \( V_2 \) has the basis
    \[
    \mathcal{B}_{V_2} = \left\{ \tilde{L}_1, \tilde{L}_2 \right\} = \left\{ \tilde{L}_{11,11}, \tilde{L}_{12,12} \right\}.
    \]
The space \( W_2 \) has the basis
    \[
    \mathcal{B}_{W_2} = \left\{ N_1, N_2 \right\} = \left\{ \mathcal{W}(12;12), \mathcal{W}(12;21) \right\}.
    \]
The following is the expression of the basis vectors of \( W_2 \) 
    \begin{align*}
    N_1 &= 2 \tilde{L}_1 + 2 \tilde{L}_2 \\
    N_2 &= -2 \tilde{L}_1 + 4 \tilde{L}_2
    \end{align*}    

\noindent When \( f = 3 \):
    \[
    \dim V_3 = \dim W_3 = 5,
    \] 

\noindent The space \( V_3 \) has the basis
    \[
    \mathcal{B}_{V_3} = \left\{ \tilde{L}_1, \tilde{L}_2, \tilde{L}_3, \tilde{L}_4 , \tilde{L}_5 \right\} = \left\{ \tilde{L}_{121,121}, \tilde{L}_{121,112},  \tilde{L}_{112,121}, \tilde{L}_{112,112}, \tilde{L}_{111,111}\right\}.
    \]

\noindent The space \( W_3 \) has the basis
\[
\begin{aligned}
\mathcal{B}_{W_3} = \{\, &N_1, N_2, \ldots, N_{5} \} \\
  & = \{ \mathcal{W}(123;123), \mathcal{W}(123;132), \mathcal{W}(123;213), \mathcal{W}(123;231), \mathcal{W}(123;312) \}.
\end{aligned}
\]

\noindent The following provides the expression of the basis vectors of \( W_3 \) 
\begin{align*}
N_1 &= 4 \tilde{L}_1 + 2 \tilde{L}_2 + 2 \tilde{L}_3 + 4 \tilde{L}_4, \\
N_2 &= 2 \tilde{L}_1 + 4 \tilde{L}_2 + 4 \tilde{L}_3 + 2 \tilde{L}_4, \\
N_3 &= -4 \tilde{L}_1 - 2 \tilde{L}_2 - 2 \tilde{L}_3 + 2 \tilde{L}_4 + 6 \tilde{L}_5, \\
N_4 &= -2 \tilde{L}_1 + 2 \tilde{L}_2 - 4 \tilde{L}_3- 2 \tilde{L}_4 + 6 \tilde{L}_5, \\
N_5 &= -2 \tilde{L}_1 - 4 \tilde{L}_2 + 2 \tilde{L}_3 - 2 \tilde{L}_4 + 6 \tilde{L}_5.
\end{align*}

\noindent When \( f = 4 \):
    \[
    \dim V_4 =15, \quad \dim W_4 = 14
    \] 
To isolate the missing invariant of \( W_4 \), we examine all known basis elements of \( V_4 \).  The space \( V_4 \) has the basis
\[
\begin{aligned}
\mathcal{B}_{V_4} = \{\, &\tilde{L}_1, \tilde{L}_2, \ldots, \tilde{L}_{15} \}= \\
& \{\tilde{L}_{1111,1111}, \tilde{L}_{1111,2222}, \tilde{L}_{1112,1112}, \tilde{L}_{1112,1121}, \tilde{L}_{1112,1211}, \\
& \tilde{L}_{1121,1112}, \tilde{L}_{1121,1121}, \tilde{L}_{1121,1211}, \tilde{L}_{1122,1122}, \tilde{L}_{1122,1212}, \\
& \tilde{L}_{1211,1112}, \tilde{L}_{1211,1121}, \tilde{L}_{1211,1211}, \tilde{L}_{1212,1122}, \tilde{L}_{1212,1212}  \}.
\end{aligned}
\]

The 14 bases of \( W_4 \) are:
\[
\begin{aligned}
\mathcal{B}_{W_4} = \{\, &N_1, N_2, \ldots, N_{14} \} = \\
& \{\mathcal{W}(1234;1234), \mathcal{W}(1234;1243), \mathcal{W}(1234;1324), \\
&\mathcal{W}(1234;1342), \mathcal{W}(1234;1423), \mathcal{W}(1234;2134), \\
&\mathcal{W}(1234;2143), \mathcal{W}(1234;2314),\mathcal{W}(1234;2341),\\
&\mathcal{W}(1234;2413), \mathcal{W}(1234;3124), \mathcal{W}(1234;3142),\\
& \mathcal{W}(1234;3412), \mathcal{W}(1234;4123)  \}.
\end{aligned}
\]

\noindent It is clear that, there exists at least one polynomial in \( \mathcal{B}_{V_4} \) which is linearly independent of all elements of \( \mathcal{B}_{W_4} \). Using SageMath \cite{sagemath}, we found that any polynomial from the set \( \mathcal{B}_{V_4} \) can be used to obtain a dimension of 15, and thus is considered as the missing invariant of \( W_4 \). We have the following expression for the basis vectors of \( W_4 \) :
\begin{align*}
N_1 &= 4 \tilde{L}_7 + 2 \tilde{L}_8 + 4 \tilde{L}_9 + 2 \tilde{L}_{10} + 2 \tilde{L}_{12} + 4 \tilde{L}_{13} + 2 \tilde{L}_{14} + 4 \tilde{L}_{15} \\
N_2 &= -6 \tilde{L}_1 + 12 \tilde{L}_2 + 6 \tilde{L}_3 + 6 \tilde{L}_4 + 6 \tilde{L}_5 + 6 \tilde{L}_6 + 2 \tilde{L}_7 + 4 \tilde{L}_8 + 2 \tilde{L}_9 - 2 \tilde{L}_{10} + 6 \tilde{L}_{11} \\
     &\quad + 4 \tilde{L}_{12} + 8 \tilde{L}_{13} - 2 \tilde{L}_{14} - 4 \tilde{L}_{15} \\
N_3 &= 2 \tilde{L}_7 + 4 \tilde{L}_8 + 2 \tilde{L}_9 + 4 \tilde{L}_{10} + 4 \tilde{L}_{12} + 2 \tilde{L}_{13} + 4 \tilde{L}_{14} + 2 \tilde{L}_{15} \\
N_4 &= -6 \tilde{L}_1 + 12 \tilde{L}_2 + 6 \tilde{L}_3 + 6 \tilde{L}_4 + 6 \tilde{L}_5 + 6 \tilde{L}_6 + 4 \tilde{L}_7 + 2 \tilde{L}_8 - 2 \tilde{L}_9 + 2 \tilde{L}_{10} + 6 \tilde{L}_{11} \\ 
     &\quad + 8 \tilde{L}_{12} + 4 \tilde{L}_{13} - 4 \tilde{L}_{14} - 2 \tilde{L}_{15} \\
N_5 &= -6 \tilde{L}_1 + 12 \tilde{L}_2 + 6 \tilde{L}_3 + 6 \tilde{L}_4 + 6 \tilde{L}_5 + 6 \tilde{L}_6 + 4 \tilde{L}_7 + 8 \tilde{L}_8 - 2 \tilde{L}_9 - 4 \tilde{L}_{10} + 6 \tilde{L}_{11} \\
     &\quad + 2 \tilde{L}_{12} + 4 \tilde{L}_{13} + 2 \tilde{L}_{14} - 2 \tilde{L}_{15} \\
N_6 &= 6 \tilde{L}_1 + 6 \tilde{L}_3 + 2 \tilde{L}_7 - 2 \tilde{L}_8 + 2 \tilde{L}_9 - 2 \tilde{L}_{10} - 2 \tilde{L}_{12} - 4 \tilde{L}_{13} - 2 \tilde{L}_{14} - 4 \tilde{L}_{15} \\
N_7 &= 12 \tilde{L}_1 - 12 \tilde{L}_2 - 6 \tilde{L}_3 - 6 \tilde{L}_5 - 2 \tilde{L}_7 - 4 \tilde{L}_8 + 4 \tilde{L}_9 + 2 \tilde{L}_{10} - 6 \tilde{L}_{11} - 4 \tilde{L}_{12} - \\
     &\quad 8 \tilde{L}_{13} + 2 \tilde{L}_{14} + 4 \tilde{L}_{15}\\
N_8 &= 6 \tilde{L}_1 + 6 \tilde{L}_3 - 2 \tilde{L}_7 - 4 \tilde{L}_8 - 2 \tilde{L}_9 - 4 \tilde{L}_{10} + 2 \tilde{L}_{12} - 2 \tilde{L}_{13} + 2 \tilde{L}_{14} - 2 \tilde{L}_{15} \\
N_9 &= 6 \tilde{L}_1 - 6 \tilde{L}_5 + 6 \tilde{L}_6 + 2 \tilde{L}_7 - 2 \tilde{L}_8 - 4 \tilde{L}_9 - 2 \tilde{L}_{10} + 4 \tilde{L}_{12} - 4 \tilde{L}_{13} - 2 \tilde{L}_{14} + 2 \tilde{L}_{15} \\
N_{10} &= 12 \tilde{L}_1 - 12 \tilde{L}_2 - 6 \tilde{L}_3 - 6 \tilde{L}_5 - 6 \tilde{L}_6 - 4 \tilde{L}_7 - 8 \tilde{L}_8 + 2 \tilde{L}_9 + 4 \tilde{L}_{10} \\
     &\quad - 2 \tilde{L}_{12} - 4 \tilde{L}_{13} + 4 \tilde{L}_{14} + 2 \tilde{L}_{15} 
\end{align*}
\begin{align*}
N_{11} &= 6 \tilde{L}_1 + 6 \tilde{L}_3 - 2 \tilde{L}_7 + 2 \tilde{L}_8 - 2 \tilde{L}_9 + 2 \tilde{L}_{10} - 4 \tilde{L}_{12} - 2 \tilde{L}_{13} - 4 \tilde{L}_{14} - 2 \tilde{L}_{15} \\
N_{12} &= 12 \tilde{L}_1 - 12 \tilde{L}_2 - 6 \tilde{L}_3 - 6 \tilde{L}_4 - 4 \tilde{L}_7 - 2 \tilde{L}_8 + 2 \tilde{L}_9 + 4 \tilde{L}_{10} - 6 \tilde{L}_{11} - 8 \tilde{L}_{12}\\
     &\quad - 4 \tilde{L}_{13} + 4 \tilde{L}_{14} + 2 \tilde{L}_{15} \\
N_{13} &= 12 \tilde{L}_1 - 12 \tilde{L}_2 - 6 \tilde{L}_3 - 6 \tilde{L}_4 - 6 \tilde{L}_6 - 8 \tilde{L}_7 - 4 \tilde{L}_8 + 4 \tilde{L}_9 + 2 \tilde{L}_{10} - 4 \tilde{L}_{12}\\
     &\quad - 2 \tilde{L}_{13} + 2 \tilde{L}_{14} + 4 \tilde{L}_{15} \\
N_{14} &= 6 \tilde{L}_1 + 6 \tilde{L}_4 + 2 \tilde{L}_7 + 4 \tilde{L}_8 - 4 \tilde{L}_9 - 2 \tilde{L}_{10} - 6 \tilde{L}_{11} - 2 \tilde{L}_{12} - 4 \tilde{L}_{13} - 2 \tilde{L}_{14} + 2 \tilde{L}_{15}
\end{align*}

\noindent When \( f = 5 \): 
    \[
    \dim V_5 =51, \quad \dim W_5 = 42
    \] 

\noindent The space \( V_5 \) has the following basis: 
\[
\begin{aligned}
\mathcal{B}_{V_5} = \{\, &\tilde{L}_1, \tilde{L}_2, \ldots, \tilde{L}_{51} \} = \\\{
& \tilde{L}_{11111,11111}, \tilde{L}_{11111,12222}, \tilde{L}_{11111,21222}, \tilde{L}_{11111,22122}, \tilde{L}_{11111,22212}, \tilde{L}_{11111,22221}, \\
& \tilde{L}_{11112,11112}, \tilde{L}_{11112,11121},  \tilde{L}_{11112,11211}, \tilde{L}_{11112,12111}, \tilde{L}_{11112,21111}, \tilde{L}_{11121,11112},\\
& \tilde{L}_{11121,11121}, \tilde{L}_{11121,11211}, \tilde{L}_{11121,12111}, \tilde{L}_{11121,21111}, \tilde{L}_{11122,11122}, \tilde{L}_{11122,11212}, \\
& \tilde{L}_{11122,11221}, \tilde{L}_{11122,12112},  \tilde{L}_{11122,12121}, \tilde{L}_{11211,11112}, \tilde{L}_{11211,11121}, \tilde{L}_{11211,11211}, \\
& \tilde{L}_{11211,12111}, \tilde{L}_{11211,21111}, \tilde{L}_{11212,11122}, \tilde{L}_{11212,11212},  \tilde{L}_{11212,11221}, \tilde{L}_{11212,12112}, \\
& \tilde{L}_{11212,12121}, \tilde{L}_{11221,11122},  \tilde{L}_{11221,11212}, \tilde{L}_{11221,11221}, \tilde{L}_{11221,12112}, \tilde{L}_{11221,12121},\\
&  \tilde{L}_{12111,11112}, \tilde{L}_{12111,11121}, \tilde{L}_{12111,11211}, \tilde{L}_{12111,12111}, \tilde{L}_{12111,21111}, \tilde{L}_{12112,11122}, \\
& \tilde{L}_{12112,11212}, \tilde{L}_{12112,11221},  \tilde{L}_{12112,12112}, \tilde{L}_{12112,12121}, \tilde{L}_{12121,11122}, \tilde{L}_{12121,11212},\\
&  \tilde{L}_{12121,11221}, \tilde{L}_{12121,12112}, \tilde{L}_{12121,12121}\}.
\end{aligned}
\]
    
\noindent We find the following basis for \( W_5 \):
\[
\begin{aligned}
\mathcal{B}_{W_5} = \{\, &N_1, N_2, \ldots, N_{42} \} = \\\{ 
& \mathcal{W}(12345;12345), \mathcal{W}(12345;12354), \mathcal{W}(12345;12435), \mathcal{W}(12345;12453),\\
& \mathcal{W}(12345;12534), \mathcal{W}(12345;13245), \mathcal{W}(12345;13254), \mathcal{W}(12345;13425),\\
& \mathcal{W}(12345;13452), \mathcal{W}(12345;13524), \mathcal{W}(12345;14235), \mathcal{W}(12345;14253),\\
& \mathcal{W}(12345;14523), \mathcal{W}(12345;15234), \mathcal{W}(12345;21345), \mathcal{W}(12345;21354),\\
& \mathcal{W}(12345;21435), \mathcal{W}(12345;21453), \mathcal{W}(12345;21534), \mathcal{W}(12345;23145),\\
& \mathcal{W}(12345;23154), \mathcal{W}(12345;23415), \mathcal{W}(12345;23451), \mathcal{W}(12345;23514), \\
&\mathcal{W}(12345;24135), \mathcal{W}(12345;24153), \mathcal{W}(12345;24513), \mathcal{W}(12345;25134),\\
& \mathcal{W}(12345;31245), \mathcal{W}(12345;31254), \mathcal{W}(12345;31425), \mathcal{W}(12345;31452),\\
& \mathcal{W}(12345;31524), \mathcal{W}(12345;34125), \mathcal{W}(12345;34152), \mathcal{W}(12345;34512),\\
& \mathcal{W}(12345;35124), \mathcal{W}(12345;41235), \mathcal{W}(12345;41253), \mathcal{W}(12345;41523),\\
& \mathcal{W}(12345;45123), \mathcal{W}(12345;51234)\}.
\end{aligned}
\]

\noindent We determine the dimension by considering all polynomials in the sets \(\mathcal{B}_{V_5}\) and \(\mathcal{B}_{W_5}\), which total 51, implying that 9 additional linearly independent invariants are required to complete a basis of \(V_5\). To achieve this dimension, we select the following polynomials from \(V_5\), ensuring that all elements are linearly independent with the basis vectors of \(W_5\). We picked the following polynomials. 
\[
\begin{aligned}
& \tilde{L}_{11111,11111},  \tilde{L}_{11111,12222},  \tilde{L}_{11111,21222}, \\
& \tilde{L}_{11111,22122},  \tilde{L}_{11111,22212},  \tilde{L}_{11112,11112},\\
& \tilde{L}_{11121,11112},  \tilde{L}_{11211,11112},  \tilde{L}_{12111,11112}.
\end{aligned}
\]

\noindent The expression of each of the 42 basis elements of \( W_5 \), as a linear combinations of the basis elements of \( V_5 \) is too large to give here. The reader can find them in \cite{selim}.

\begin{thm}
    For \( f \leq 5 \) any basis vectors \( N_i \) of \( W_f \) is a linear combination of the basis vectors of \( V_f \) with integer coefficients (even numbers) for corresponding $f$. i.e.
    \[
    N_i = \sum_{i} a_i \, \tilde{L}_i \quad \text{where } a_i \in 2\mathbb{Z} \text{ and } \tilde{L}_i \text{ are the basis of } V_f.
    \]
\end{thm}

\textbf{Acknowledgements}.
The first named author was supported in part by funds from  Ministry of Education, Culture, Sports, Science and Technology (MEXT), Japan and the second named author was supported by JSPS KAKENHI Grant Numbers 24K06827 and
24K06644.

Department of Mathematics, Khulna University of Engineering \& Technology, Khulna-9203, Bangladesh and
Graduate School of Natural Science and Technology, Kanazawa University, Ishikawa 920-1192, Japan

\textit{Email address}: \texttt{selim\_1992@math.kuet.ac.bd}

\bigskip

Faculty of Mathematics and Physics,
Institute of Science and Engineering,
Kanazawa University,
Kakuma-machi,
Ishikawa 920-1192,
Japan

\textit{Email address}: \texttt{oura@se.kanazawa-u.ac.jp}

\bigskip

Faculty of Engineering, University of Yamanashi, 400-8511, Japan

\textit{Email address}: \texttt{mkosuda@yamanashi.ac.jp}

\end{document}